\documentclass[11pt]{amsart}
\usepackage[letterpaper,top=3.5cm,bottom=3.3cm,left=3cm,right=3cm,marginparwidth=1.75cm]{geometry}
\usepackage{graphicx}
\usepackage{amssymb}
\usepackage{hyperref}
\numberwithin{equation}{section}
\numberwithin{figure}{section}

\usepackage{amsthm}
\usepackage{mathrsfs}
\usepackage{tabularx}
\usepackage{mathtools}
\usepackage{bbold}
\usepackage{indentfirst}
\usepackage{enumitem} 
\usepackage{tikz}
\usetikzlibrary{intersections, calc, decorations.markings, arrows.meta}
\usepackage{mella-commands}

\theoremstyle{plain}
\newtheorem{theoL}{Theorem}

\newtheorem{definL}[theoL]{Definition}
\newtheorem{propL}[theoL]{Proposition}

\newtheorem{theoPerso}{Theorem}

\newtheorem{propPerso}[theoPerso]{Proposition}
\newtheorem*{propB}{Proposition B}
\newtheorem*{propC}{Proposition C}
\newtheorem*{propD}{Proposition D}

\newtheorem{cor}[theoPerso]{Corollary}
\newtheorem{lem}[theoPerso]{Lemma}
\newtheorem{defin}[theoPerso]{Definition}
\newtheorem*{definA}{Definition A}

\theoremstyle{remark} 
\newtheorem*{rem}{Remark}

\title[Genericity of ergodic measures with flat strips]{Genericity of ergodic measures on non-positively curved manifolds with flat strips}

\author{Paul Mella}
\address{Sorbonne Université, LPSM}
\email{paul.mella@sorbonne-universite.fr}

\keywords{ergodic theory, smooth dynamics, geodesic flow, non-positive curvature}
\subjclass[2020]{37A10, 37D25, 37D40}

\begin{document}

\begin{abstract}
We study the geodesic flow defined on the unit tangent bundle of a Riemannian manifold with non-positive curvature. In the case of surfaces, Coudène and Schapira proved that the existence of a flat strip prevents the set of ergodic measures from being generic in the space of invariant probability measures. We show the existence of a three-dimensional manifold that contains an embedded flat torus and where ergodicity is a generic property. More generally, Coudène and Schapira proved that ergodicity is a generic property within the space of invariant probability measures defined on $\Omega_{NF}$, the set of non-wandering vectors whose orbits do not bound a flat strip. We introduce a larger set $\mathcal{R}_w$ where the same property holds, and we illustrate the extent of our findings using a famous example of a rank-one manifold attributed to Gromov.
\end{abstract}


\maketitle

\section{Introduction}

The geodesic flow defined on the unit tangent bundle of a negatively curved compact Riemannian manifold is one of the first examples of uniformly hyperbolic dynamics \cite{A67}. Following the pioneering works of Anosov, numerous authors have delved deeper into the connection between the dynamics of the geodesic flow and the geometry of the underlying manifold. For example, on a compact and connected negatively curved manifold, the geodesic flow satisfies the Closing Lemma \cite{A67}, it is topologically transitive \cite{E73}, and the Liouville measure is ergodic \cite{A67, AS67, H39}. The weak-$\ast$ topology of the space of invariant probability measures for the geodesic flow has been extensively studied in this setting. Sigmund notably proved that ergodicity is a generic property in the space of invariant probability measures \cite{S72}, which makes the latter a realization of the Poulsen simplex \cite{LOS78}. 

\medskip

The notion of flat strip is crucial when studying the properties of the geodesic flow on non-positively curved manifolds. Let $M$ be a connected, geodesically complete Riemannian manifold of non-positive curvature. Throughout this work, the unit tangent bundle of $M$ will be denoted by $T^1M$ and the geodesic flow by $(g^t)_{t\in\mathbb{R}}$. A \textit{flat strip} is a totally geodesic isometric immersion of $\mathbb{R}\times I$ in $M$, with $I\subset\R$ being a closed interval of positive width. While in negative curvature every geodesic is expanding, in the sense that nearby initially parallel trajectories exhibit exponential divergence in forward time, in non-positive curvature, geodesics that bound a flat strip may not have this property. In particular, their \textit{rank}, which is defined as the dimension of the vector space of parallel Jacobi fields along a geodesic, is at least two. While compact non-positively curved manifolds where all geodesics have rank at least two are locally symmetric according to the Rank Rigidity Theorem \cite{BGS85}, it appears that manifolds containing geodesics that do not bound a flat strip exhibit properties close to those of negatively curved manifolds.

Let us denote by $\ONF$ the subset of $T^1M$ containing all non-wandering vectors whose orbits do not bound a flat strip, borrowing the notation from Coudène and Schapira. We also introduce the notation $P_{NF}$ for the set of periodic vectors in $\ONF$. When $M$ is compact, every vector is non-wandering, thus $\ONF$ contains the set of vectors of rank one, also called \textit{regular vectors}.

Coudène and Schapira proved in 2014 that the geodesic flow is transitive, satisfies the Closing Lemma, and has a local product structure in restriction to $\ONF$ \cite{CS14}. As a consequence of those purely dynamical properties, the \textit{Dirac measures} on $P_{NF}$, that is to say, the invariant probability measures supported by periodic orbits, are dense in the space of all invariant probability measures on $\ONF$, equipped with the weak-$\ast$ topology. Using this density result and the fact that ergodic measures are a $G_\delta$ subset of this topological space \cite{P61}, Coudène and Schapira proved the genericity of ergodic measures.

This bridge between the dynamical properties of a flow and the genericity of ergodic measures was generalized in 2016 by Gelfert and Kwietniak \cite{GK16}. For any subset $S$ of a dynamical system $X$, let us denote by $\mathcal{M}^1(S)$ (resp. $\mathcal{M}^1_e(S)$ and $\mathcal{M}^1_{per}(S)$) the set of invariant (resp. ergodic and Dirac) measures on $S$, which can be identified with the set of invariant (resp. ergodic and Dirac) measures on $X$ that give full measure to $S$. For a set of periodic vectors $S\subset X$, Gelfert and Kwietniak introduce two dynamical properties: the \textit{linkability} of $S$, which is a periodic analogue of the well-known specification property, and the $S$-\textit{closeability} of $X$, which we will define in Section \ref{sec:cdirac}. In terms of topological properties of the space of invariant measures, the \textit{linkability} of $S$ implies the density of $\mathcal{M}^1_{per}(S)$ in its convex hull, while the $S$-\textit{closeability} of $X$ implies the density of $\mathcal{M}^1_{per}(S)$ in $\mathcal{M}_{e}^1(X)$. They show that if $X$ is a Polish space, then $\mathcal M^1_{per}(X)$ is dense and $\mathcal M^1_{e}(X)$ is generic in $\mathcal M^1(X)$ as soon as there exists a set of periodic points satisfying the linkability and closeability properties. Their results focus on discrete-time dynamical systems but can also be applied to continuous-time dynamical systems. Notably, the main arguments of the proof presented by Coudène and Schapira in \cite{CS14} can be restated as the linkability of $P_{NF}$ and the $P_{NF}$-closeability of $\ONF$. 

\medskip

The present article aims to generalize those results to a larger subset than $\ONF$, denoted by $\A$. We will show that $\A$ is a $G_\delta$-subset of $T^1M$ and that it is $P_{NF}$-closeable. Then, using the linkability of $P_{NF}$ proved by Coudène and Schapira, and the results of Gelfert and Kwietniak, we will be able to prove Proposition \ref{gdirac}. The elements of $\A$ will be called \textit{weakly regular vectors} to highlight the fact that they carry most of the properties of regular vectors, especially regarding the topological properties of the invariant measures they support.

For any vector $v\in T^1M$ and $\epsilon>0$, the $\epsilon$-local stable and unstable sets of $v$ are denoted by $W^s_\epsilon(v)$ and $W^u_\epsilon(v)$. Their definition is standard: $$W^s_\epsilon(v) = \left\{u\in T^1 M \mid \forall t\ge 0, ~d\left(g^tu,g^tv\right) < \epsilon\right\}$$
$$W^u_\epsilon(v) = \left\{u\in T^1 M \mid \forall t\le 0, ~d\left(g^tu,g^tv\right) < \epsilon\right\}$$   

The following definition is given only in the case of compact manifolds but will be generalized to the non-compact case in Section \ref{sec:weakly_regular}.

\begin{definL}\label{def:weakly_regular}
Let $M$ be a compact manifold. A vector $v \in T^1M$ will be said to be weakly regular if for all $\epsilon>0$, there exist two non-negative numbers $\tau^s,\tau^u\ge0$ such that both $W^s_\epsilon\big(g^{\tau^s}(v)\big)$ and $W^u_\epsilon\big(g^{-\tau^u}(v)\big)$ intersect $\ONF$. The set of weakly regular vectors will be denoted by $\A$.
\end{definL}

In Section \ref{sec:cdirac}, we prove the density of Dirac measures and the genericity of ergodic measures on $\A$, using the results of Coudène and Schapira \cite{CS14} and Gelfert and Kwietniak \cite{GK16}.

\begin{propL}\label{gdirac}Let $M$ be a connected, geodesically complete manifold whose curvature is non-positive and bounded below by a negative constant, such that $\ONF$ is open in $T^1M$ and contains at least three periodic orbits. Then $\mathcal{M}^1_e(\A)$ is a generic subset of $\mathcal{M}^1(\A)$ and the following inclusion holds:
    $$\mathcal{M}^1(\A) \subset \overline{\mathcal{M}^1_{per}(P_{NF})}$$ 
\end{propL}

Our result generalizes the main theorem from \cite{CS14}, according to which $\mathcal{M}^1_e(\ONF)$ is generic in $\mathcal{M}^1(\ONF)$ and $\mathcal{M}^1_{per}(P_{NF})$ is dense in $\mathcal{M}^1(\ONF)$. It also narrows the gap with the conclusion of \cite{M26}, where a necessary condition on $M$ for $\mathcal{M}^1_e(T^1M)$ to be generic in $\mathcal{M}^1(T^1M)$ is introduced. However, the question of whether $\A$ is the maximal subset of $T^1M$ that satisfies Proposition \ref{gdirac}, i.e., whether $\mu(\A)=1$ is a necessary condition for an ergodic probability measure to be approximated by Dirac measures supported on $P_{NF}$, remains open to this day. Another closely related open question would be to describe precisely the set $\overline{\mathcal{M}^1_{per}(P_{NF})}$.

\medskip

In Section \ref{sec:application}, our new result is applied to two notable examples of non-positively curved manifolds of dimension $3$.

Proposition \ref{gdirac} can be used to show the existence of a manifold containing a flat torus (and thus also a flat cylinder) where the genericity of ergodic probability measures holds. This manifold is a well-known example attributed to Heintze. Here, it highlights how different the case of $3$-dimensional manifolds is compared to the case of surfaces, where the existence of a flat strip prevented $\mathcal{M}^1_e(T^1M)$ from being dense \cite{CS11, CS14}. 

\begin{propL}\label{prop:C}There exists a compact, non-positively curved manifold $M$ of dimension $3$ that contains a totally geodesic flat torus such that the set of ergodic measures is generic in the space of invariant probability measures.
\end{propL}

Proposition \ref{gdirac} is applied in Section \ref{sec:ex} to a well-known example of a rank-one manifold that was introduced by Gromov \cite{G78}. It is particularly interesting in our case because the inclusion $\ONF\subset\A$ is proper. We will prove that, in this setting, Proposition \ref{gdirac} is an equivalence.

\begin{propL}\label{prop:D}Let $M$ be the Gromov manifold defined in Section \ref{sec:ex}. An ergodic probability measure $\mu$ on $T^1 M$ can be approximated by Dirac measures supported on orbits that do not bound a flat strip if and only if $\mu(\A)=1$.
\end{propL}

Additionally, we illustrate how Proposition \ref{gdirac} can be used to obtain a better understanding of the set of finite invariant measures on the unit tangent bundle of the Gromov manifold: we give a precise and elementary description of $\mathcal{M}^1(T^1M)$, $\mathcal{M}^1_e(T^1M)$, and $\overline{\mathcal{M}^1_e(T^1M)}$ in this setting. It would be interesting for future studies to explore to what extent the descriptions of those sets can be generalized to the general setting of rank-one manifolds.

\medskip

\textbf{Acknowledgments.} This work would not have been possible without my PhD supervisor, Yves Coudène, who introduced me to this area of study. I am deeply grateful for his guidance and feedback throughout my research and the writing of this article.

\section{The geodesic flow in restriction to the set of weakly regular vectors}
\label{sec:cdirac}

Throughout Section \ref{sec:cdirac}, $M$ will denote a connected and geodesically complete Riemannian manifold of non-positive curvature. For any vector $v\in T^1M$ and $\delta>0$, the open geodesic ball of radius $\delta$ centered at $v$ will be denoted $B_\delta(v)$.

In this section, we prove that every invariant probability measure on $\A$ can be approximated by Dirac measures supported on $P_{NF}$, and that the set of ergodic measures $\mathcal{M}^1_e(\A)$ is a generic subset of $\mathcal{M}^1(\A)$.

\subsection{Weak stable and unstable sets}

Most of the recent advancements in the study of non-positively curved manifolds stem from the theory of Hadamard manifolds. The universal cover $\M$ of $M$ is a Hadamard manifold. We will denote its unit tangent bundle by $T^1\M$ and the canonical projection on $\M$ also by $\pi:T^1\M\rightarrow\M$ since there is no ambiguity with $\pi:T^1M\rightarrow M$. The projection from $T^1\tilde M$ to $T^1M$ will be denoted by $\mathrm{Proj}:T^1\tilde M\to T^1M$. Following the notations used in \cite{B95}, the visual boundary of $\M$ will be denoted by $\M(\infty)$ and for any vector $\tilde{v}\in T^1\M$, the notations $\tilde{v}(-\infty)$ and $\tilde{v}(+\infty)$ will denote the two endpoints of the geodesic generated by $\tilde{v}$. Two vectors $\tilde u , ~\tilde v \in T^1\tilde M$ are said to be \textit{positively asymptotic} (resp. \textit{negatively asymptotic}) if their orbits have the same endpoint at $+\infty$ (resp. $-\infty$), and the set of all vectors in $T^1\tilde M$ that are  positively asymptotic (resp. negatively asymptotic) to $\tilde v$ is called the \textit{stable leaf} (resp. \textit{unstable leaf}) and denoted by $W^{so}(\tilde v)$ (resp. $W^{uo}(\tilde v)$). 

The next lemma will allow us to simplify the definitions in Section \ref{sec:loc_stable} and Section \ref{sec:weakly_regular} in the case where $M$ is compact.

\begin{lem}\label{lem:lift}
    Let $v\in T^1M$ such that the injectivity radius of $T^1M$ is bounded below by $\epsilon_0$ along the half-orbit $\{g^tv\mid t\ge0\}$ (resp. $\{g^tv\mid t\le0\}$). Then, for all $\epsilon < \epsilon_0$ and all $u\in W_\epsilon^s(v)$ (resp. $u\in W_\epsilon^u(v)$), there exist lifts $\tilde u, \tilde v$ of $u, v$ such that $\tilde u\in W_\epsilon^s(\tilde v)$ (resp. $\tilde u\in W_\epsilon^u(\tilde v)$).
\end{lem}
\begin{rem}
    The assumption that the injectivity radius of $T^1M$ is bounded below by a positive constant along the half-orbit $\{g^tv\mid t\ge0\}$ or $\{g^tv\mid t\le0\}$ holds for every $v\in T^1M$ as soon as $M$ is compact.
\end{rem}

\begin{proof}
Let $\tilde v\in T^1\tilde M$ be a lift of $v$. Let $u\in W_\epsilon^{s}(v)$ and for all $t\ge0$, let us denote by $\delta(t)$ the distance between $g^tu$ and $g^tv$, which is smaller than $\epsilon$.

Let $t\ge 0$. Since $\epsilon<\epsilon_0$, the exponential map is an embedding of the $\epsilon$-ball centered at $0$ in the tangent space $T_{g^tv}T^1M$ into $T^1M$. In particular, its inverse is continuous. Moreover, since $\delta(t)<\epsilon$, the exponential map defined on the tangent space $T_{g^tu}T^1M$ is also an embedding when restricted to a neighborhood of $0$, which implies that the following is a smooth map on a neighborhood of $(g^tu,g^tv)$:
$$\begin{aligned}
    T^1M\times T^1M&\to TT^1M\\
    (p, q)&\mapsto \exp_p^{-1}(q)
\end{aligned}$$

This holds for all $t\ge0$. Hence, the following is also a continuous map:
$$\begin{aligned}
    X:~[0, +\infty) &\to TT^1M\\
    t~~~~&\mapsto (\exp_{g^tv})^{-1}\big(g^tu\big)
\end{aligned}$$

For all $t\ge0$, one can lift the vector $X(t)$ to a vector $\tilde X_t\in T_{g^t\tilde v}T^1\tilde M$. The family $(\tilde u_t)_{t\ge0}$ defined by:
$$\forall t\ge0, \qquad \tilde u_t = \exp_{g^t\tilde v}(\tilde X_t)$$
is continuous. For all $t\ge0$, since $\mathrm{Proj}$ is a local isometry, $\tilde u_t$ is a lift of $g^tu$, and since all the lifts of $g^tu$ form a discrete set, the family $(\tilde u_t)_{t\ge0}$ must be exactly the orbit of the vector $\tilde u_0$ by the geodesic flow. Finally, the proof is complete:
$$\forall t\ge0, \quad d\big(g^t\tilde v, g^t\tilde u_0\big) =\|\tilde X_t\| = \delta(t) < \epsilon$$
\end{proof}

Conversely, considering two vectors $u$ and $v$ that generate two geodesics $\gamma_u$ and $\gamma_v$, the distance $d(g^tu,g^tv)$ is not bounded above by a multiple of $d(\gamma_u(t), \gamma_v(t))$ in general. However, comparing the Sasaki geodesic distance with the distance on the base-manifold is possible in the universal cover, in restriction to the stable leaf of a given vector: this is established by Lemma \ref{lem:sasaki} and Corollary \ref{cor:sasaki}.

\medskip

According to \cite{B95}, for all $\tilde v \in T^1\tilde M$, there exists a function $f\in\mathcal{C}^2(\M)$, called \textit{Busemann function}, such that the stable leaf of $\tilde v$ is the image of the gradient of $-f$:

$$\forall \tilde u\in T^1\M, \quad  \tilde{u}(+\infty)=\tilde{v}(+\infty) \iff \tilde u= -\nabla f(\pi(\tilde u))$$

The following lemma is a direct consequence of the properties that are presented in \cite{B95}.

\begin{lem}\label{lem:sasaki}
Assume that the curvature of $\M$ is bounded below by a negative constant $-b^2$. Then the gradient of every Busemann function is $\sqrt{1+b^2}$-Lipschitz with respect to the geodesic distance induced by the Sasaki metric.
\end{lem}
\begin{rem}
    The assumption that the curvature of $\tilde M$ is bounded below by $-b^2$ holds as soon as $M$ is compact.
\end{rem}
\begin{proof}
Let $f$ be a Busemann function and $x, y\in \tilde M$. Let us write $\tau=d(x,y)$ and denote by $(\gamma(s))_{0\leq s\leq\tau}$ the geodesic in $\M$ joining $x$ and $y$. 

Let $s\in[0,\tau]$. Let $z=\gamma(s)$ and $S^{\perp}(s)\in T_z\tilde M$ be the projection of $\dot\gamma(s)$ on the subspace orthogonal to $\nabla f(z)$ in $T_z\tilde M$, so that $\dot\gamma(s) - S^{\perp}(s)$ is collinear to $\nabla f(z)$. According to \cite[Proposition 3.2]{B95}, the vector field $\nabla f$ is continuously differentiable, and its covariant derivative along $S^\perp(s)$ is given by:
$$D_{S^\perp(s)}(\nabla f)(z) = -J'_{S^\perp(s)}(0)$$
where $J_{S^\perp(s)} : \mathbb{R}\rightarrow T\M$ is the stable Jacobi field along the geodesic generated by $\nabla f(z)$, determined by the initial condition $J_{S^\perp(s)}(0)={S^\perp(s)}$. Then, \cite[Proposition 2.9]{B95} shows that $\big\|J_{S^\perp(s)}'(0)\big\|\leq b\big\|J_{S^\perp(s)}(0)\big\|$. Moreover, since $\dot\gamma(s)-S^\perp(s)$ is collinear to $\nabla f(z)$, and the integral curves of $\nabla f$ are geodesics, the geodesic equation yields $D_{\dot\gamma(s)-S^\perp(s)}(\nabla f)(z)=0$. Hence,
$$\big\|D_{\dot\gamma(s)}(\nabla f)(z)\big\|\leq b\big\|J_{S^\perp(s)}(0)\big\| \le b$$

On the other hand, $\nabla f\circ \gamma$ is a $\mathcal{C}^1$ path in $T\M$ joining $\nabla f(x)$ and $\nabla f(y)$, thus its length is bounded below by the geodesic distance induced by the Sasaki metric. Decomposing its derivative as the sum of its horizontal component $\dot \gamma$ and its vertical component $D_{\dot\gamma}(\nabla f \circ \gamma)$, one has:
$$
\begin{aligned}
d(\nabla f(x),\nabla f(y)) &\leq \int_0^\tau \left\|\frac{d}{ds}(\nabla f(\gamma(s)))\right\|ds\\
&\leq \int_0^\tau \sqrt{\|\dot\gamma(s)\|^2 + \|D_{\dot\gamma(s)}(\nabla f)(\gamma(s))\|^2} \, ds \\
  &\leq \tau \sqrt{1+b^2}
\end{aligned}$$\end{proof}

\begin{cor}\label{cor:sasaki}
Assume that the curvature of $\tilde M$ is bounded below by a negative constant $-b^2$. Let $\epsilon>0$. For all $\tilde u$ and $\tilde v$ in $T^1\tilde M$ that are positively (resp. negatively) asymptotic, and such that $d(\tilde u, \tilde v)<\frac\epsilon{\sqrt{1+b^2}}$, one has:
$$\tilde u\in W^s_\epsilon(\tilde v)\quad (\text{resp.}~\tilde v\in W_\epsilon^u(\tilde u))$$
\end{cor}
\begin{proof}Let $\tilde u,\tilde v\in T^1\tilde M$ positively asymptotic, and let $f$ be a Busemann function centered at $\tilde v(+\infty)$. Then, one has $g^t\tilde u = -\nabla f\big(\pi\big(g^t\tilde u\big)\big)$ and $g^t\tilde v = -\nabla f\big(\pi\big(g^t\tilde v\big)\big)$ for all $t\ge 0$, and Lemma \ref{lem:sasaki} yields the following inequality:
$$\forall t\ge 0, \quad d\big(g^t\tilde u, g^t\tilde v\big)\le d\big(\pi\big(g^t\tilde u\big), \pi\big(g^t\tilde v\big)\big)\sqrt{1+b^2}$$

The conclusion follows from the fact that the function $s\mapsto d(\pi(g^s\tilde u), \pi(g^s\tilde v))$ is convex since $\tilde M$ is a Hadamard manifold, and uniformly bounded on the interval $[0, +\infty)$ since $\tilde u$ and $\tilde v$ are positively asymptotic:

$$\forall t\ge 0, \quad d\big(g^t\tilde u, g^t\tilde v\big)~\le~ d(\pi(\tilde u), \pi(\tilde v))\sqrt{1+b^2}~<~\epsilon$$

A similar argument yields the analogue property in the case where $\tilde u$ and $\tilde v$ are negatively asymptotic.
\end{proof}

The following result can be seen as an extension of the local product structure that holds in restriction to $\ONF$ \cite[Lemma 4.5]{CS14}. It will be used in Proposition \ref{prop:polish}, Lemma \ref{lem:equiv} and Lemma \ref{lem:closeability}. While the local product structure of $\ONF$ that was proved in \cite[Lemma 4.5]{CS14} would have been sufficient for the proof of Proposition \ref{prop:polish} and Lemma \ref{lem:equiv}, the proof of the $P_{NF}$-closeability of $\A$ (Lemma \ref{lem:closeability}) requires lifting the situation to the universal cover, which is done in the proof of Lemma \ref{lem:ELPS}.

\begin{lem}\label{lem:ELPS}Let us assume that $\ONF$ is open and the curvature of $M$ is bounded below by $-b^2$. Let $u\in\ONF$, $v\in T^1M$, and $\epsilon>0$. With the notation $\epsilon' = \frac\epsilon{\sqrt{1+b^2}}$, let us assume that there exist lifts $\tilde u, \tilde v$ of $u, v$ such that $\tilde u \in W^u_{\epsilon'}(\tilde v)$. Then, the following holds:
$$\exists \delta'>0, \quad \forall \tilde u'\in B_{\delta'}(\tilde v), \quad \exists \hat u\in W^s_\epsilon(\tilde u)\cap W^u_\epsilon(\tilde u')\cap \mathrm{Proj}^{-1}(\ONF)$$
\end{lem}
\begin{proof}

Set $\delta = \frac12\big({\epsilon' - d(\tilde u,\tilde v)}\big)$, which is positive since $\tilde u\in W_{\epsilon'}^u(\tilde v)$. Making $\delta$ smaller if necessary, let us assume that $B_{\delta}(u)$ is contained in $\ONF$: this is possible because $u\in \ONF$ and $\ONF$ is open.
Then, \cite[Lemma 3.1]{B95} allows us to find a neighborhood $\tilde{\mathcal U}$ of $\tilde u(-\infty)$ such that for all $\eta\in\tilde{\mathcal U}$, there exists a vector $\hat u\in B_{\delta}(\tilde u)$ that generates a geodesic whose endpoints are $\eta$ and $\tilde u(+\infty)$.
Moreover, by definition of the topology of $\tilde M(\infty)$, since $\tilde{\mathcal U}$ is a neighborhood of $\tilde v(-\infty)$, there exists $\delta'>0$ such that for all $\tilde u'\in B_{\delta'}(\tilde v)$, one has $\tilde u'(-\infty)\in \tilde{\mathcal U}$.
Without loss of generality, we can assume that $\delta'<\delta$.
\begin{figure}[tbp]
    \centering
\begin{tikzpicture}[scale=0.7, font=\Large]

    \begin{scope}
        \clip (-8.0, -2.1) rectangle (-6.0, 2.1);
        \draw[gray, line width=16pt] (0,0) ellipse (7 and 5);
    \end{scope}

    \draw[thick] (0,0) ellipse (7 and 5); 
    \node[anchor=east] at (-6, 2.6) {$\tilde{\mathcal U}$};

    \begin{scope}
        \clip (0,0) ellipse (6.98 and 4.98);
        \draw[thick] (-7,0) -- (7,0);
        \draw[thick] (6.5,2) .. controls (0, 0) .. (-7.0, 1.5); 
        \draw[thick] (6.5, -2.5) .. controls (-0.5, 0.1) and (-3, -0.1) .. (-7, 0); 
        \draw[thick, densely dotted] (6.5, -2.5) .. controls (2, -0.5) and (-1.5, 0) .. (-7, 1.5);
    \end{scope}

    
    \coordinate (P_w) at (-0.3, -0.2);
    \fill (P_w) circle (3pt) node[right=40pt, yshift=-3pt] {$\hat{u}$};
    \draw[->, line width=1.5pt] (P_w) -- +(1.6, -0.4); 

    \coordinate (P_upr) at (-0.3, 0.43);
    \fill (P_upr) circle (3pt) node[left=0pt, yshift=10pt] {$\tilde{u}'$}; 
    \draw[->, line width=1.5pt] (P_upr) -- +(1.8, -0.05);

    \coordinate (P_v) at (-0.3, 0);
    \fill (P_v) circle (3pt) node[left=5pt, yshift=6pt] {$\tilde{v}$}; 
    \draw[->, line width=1.5pt] (P_v) -- +(1.7, 0);

    \coordinate (P_u) at (-0.2, -0.5);
    \fill (P_u) circle (3pt) node[left=5pt, xshift=-3pt] {$\tilde{u}$};
    \draw[->, line width=1.5pt] (P_u) -- +(1.6, -0.3);

\end{tikzpicture}
\caption{Universal cover of $M$ with the orbits of $\tilde{v}$, $\tilde{u}$,  $\tilde{u}'$ and $\hat u$.}
    \label{fig:universal-cover}
\end{figure}

Let $\tilde u' \in B_{\delta'}(\tilde v)$.
Then, there exists a vector $\hat u\in B_{\delta}(\tilde u)$ that generates a geodesic whose endpoints are $\tilde u'(-\infty)$ and $\tilde u(+\infty)$. Hence,
$$\begin{cases}
    d(\hat u, \tilde u) < \delta \le \epsilon'\\
    d(\hat u, \tilde u') \le d(\hat u, \tilde u)+d(\tilde u, \tilde v)+d(\tilde v, \tilde u') < 2\delta + d(\tilde u, \tilde v) \le \epsilon'
\end{cases}$$
and by Corollary \ref{cor:sasaki}, one has:
$$\begin{cases}
    \forall t\ge 0, \quad d(g^t\hat u, g^t\tilde u)\le\epsilon\\
    \forall t\le 0, \quad d(g^t\hat u, g^t\tilde u')\le\epsilon
\end{cases}$$

Hence, $\hat u\in W^s_\epsilon(\tilde u)\cap W^u_\epsilon(\tilde u')$, and since $B_\delta(u)\subset\ONF$, it also belongs to $\mathrm{Proj}^{-1}(\ONF)$.
\end{proof}

\subsection{Local stable and unstable sets}\label{sec:loc_stable}

We introduce the following sets, which will be very useful to define the set $\A$ and to prove its dynamical properties. 

\begin{defin}\label{def:eps_loc}
For any $v\in T^1M$, the weak $\epsilon$-local stable (resp. unstable) set of $v$ will be denoted $W_\epsilon^{ws}(v)$ (resp. $W_\epsilon^{wu}(v)$) and defined as the set of vectors $u\in T^1M$ such that there exist lifts $\tilde u, \tilde v$ of $u, v$, and non-negative times $\tau, \sigma\ge0$ that satisfy:
$$g^\tau\tilde u \in W^s_\epsilon(g^\sigma\tilde v)\qquad(\text{resp.}~g^{-\tau}\tilde u \in W^u_\epsilon(g^{-\sigma}\tilde v))$$
\end{defin}

\begin{rem}Let $u, v\in T^1M$ and $\epsilon>0$. Then $u \in W_\epsilon^{ws}(v)$ if and only if $v\in W_\epsilon^{ws}(u)$. Similarly, $u \in W_\epsilon^{wu}(v)$ if and only if $v\in W_\epsilon^{wu}(u)$.
\end{rem}

\begin{propPerso}\label{prop:loc_stable_compact_case}
    If $M$ is compact, then for all $v\in T^1M$ and $\epsilon>0$ small enough:
$$W_\epsilon^{ws}(v) = \bigcup_{\tau, \sigma\ge0}g^{-\tau}\big(W^s_\epsilon(g^\sigma v)\big)\qquad\text{and}\qquad W_\epsilon^{wu}(v) =  \bigcup_{\tau, \sigma\ge0}g^{\tau}\big(W^u_\epsilon(g^{-\sigma} v)\big)$$
\end{propPerso}
\begin{proof}Since $M$ is compact, there exists a positive lower bound on the injectivity radius of $T^1M$. Let $\epsilon>0$ be smaller than that bound, and let $v\in T^1M$. Then, for all $\tau, \sigma\ge0$, the inclusions $g^{-\tau}\big(W^s_\epsilon(g^\sigma v)\big)\subset W_\epsilon^{ws}(v)$ and $g^{\tau}\big(W^u_\epsilon(g^{-\sigma} v)\big)\subset W_\epsilon^{wu}(v)$ are consequences of Lemma \ref{lem:lift}.\end{proof}



\begin{lem}\label{lem:invariant}Let $v\in T^1M$. The set $W_\epsilon^{ws}(v)$ is invariant under the geodesic flow. Moreover, for all $t\in \R$, the set $W_\epsilon^{ws}(g^tv)$ coincides exactly with $W_\epsilon^{ws}(v)$. The same properties also hold with the weak $\epsilon$-local unstable set.\end{lem}
\begin{proof}Let $u\in W_\epsilon^{ws}(v)$. Let $\tilde u, \tilde v$ be lifts of $u, v$ and $\tau, \sigma \ge 0$ be such that:
\begin{equation}\label{eq:strong-stable}\forall t\ge 0, \quad d\big(g^{t+\tau}\tilde u, g^{t+\sigma}\tilde v\big)<\epsilon\end{equation}

Let $s\in\R$, and let us prove that $g^su$ belongs to $W_\epsilon^{ws}(v)$. If $\tau-s\ge0$, since for all $t\ge0$, one has $g^{t+\tau}\tilde u = g^{t+\tau-s}(g^s\tilde u)$, Equation \ref{eq:strong-stable} means exactly that $g^su\in W_\epsilon^{ws}(v)$. If $\tau-s<0$, then $s$ is positive. Hence, $\sigma+s\ge0$ and for all $t\ge0$, since $t+s\ge0$, Equation \ref{eq:strong-stable} implies: 
$$d\big(g^{t+\tau}(g^s\tilde u), g^{t+(\sigma+s)}\tilde v\big)<\epsilon$$

Hence, $g^su$ belongs to $W_\epsilon^{ws}(v)$. We have proved that the weak $\epsilon$-local stable set is invariant. It follows from the remark below Definition \ref{def:eps_loc} that for all $t\in \R$, the set $W_\epsilon^{ws}(g^tv)$ coincides exactly with $W_\epsilon^{ws}(v)$. A similar argument shows that the same properties hold for $W_\epsilon^{wu}(v)$.
\end{proof}

\subsection{Weakly regular vectors}\label{sec:weakly_regular}

We recall from the introduction that $\ONF$ denotes the set of vectors whose orbits do not bound a flat strip. We can now rewrite the definition of the set of weakly regular vectors given in the introduction, using the weak $\epsilon$-local stable and unstable sets and generalizing it to the non-compact case.

\begin{definA}
Let $M$ be a geodesically complete and non-positively curved manifold. A vector $v \in T^1M$ will be said to be weakly regular if for all $\epsilon>0$, the two sets $W_\epsilon^{ws}(v)$ and $W_\epsilon^{wu}(v)$ intersect $\ONF$. The set of weakly regular vectors will be denoted by $\A$.
\end{definA}
\begin{rem}It follows from the remark below Definition \ref{def:eps_loc} that for all $v\in T^1M$ and $\epsilon>0$, the set $W_\epsilon^{ws}(v)$ intersects $\ONF$ if and only if there exists $u\in \ONF$ such that $v$ belongs to $W_\epsilon^{ws}(u)$. The same property also holds for $W_\epsilon^{wu}(u)$. Denoting by $W_\epsilon^{ws}(\ONF)$ (resp. $W_\epsilon^{wu}(\ONF)$) the union of all weak $\epsilon$-local stable (resp. unstable) sets of the vectors in $\ONF$, the definition of $\A$ becomes:
\begin{equation}\label{eq:intersection}\A = \bigcap_{\epsilon>0}\big(W_\epsilon^{ws}(\ONF)\cap W_\epsilon^{wu}(\ONF)\big)\end{equation}
\end{rem}

\begin{rem}Let $v\in T^1M$. The families of sets $(W_\epsilon^{ws}(v))_{\epsilon>0}$ and $(W_\epsilon^{wu}(v))_{\epsilon>0}$ are non-decreasing with respect to the parameter $\epsilon$, in the sense that:
\begin{equation}\label{eq:inclusion}
\forall \epsilon < \eta, \qquad W_\epsilon^{ws}(v)\subset W_\eta^{ws}(v)\qquad\text{and}\qquad W_\epsilon^{wu}(v)\subset W_\eta^{wu}(v)\end{equation}
\end{rem}
\medskip
    Let $v\in T^1M$ such that for all $\epsilon>0$, there exist two non-negative numbers $\tau^s,\tau^u\ge0$ such that both $W^s_\epsilon\big(g^{\tau^s}(v)\big)$ and $W^u_\epsilon\big(g^{-\tau^u}(v)\big)$ intersect $\ONF$. If, additionally, the injectivity radius of $T^1M$ is bounded below by a positive constant $\epsilon_0$ along the half-orbits $\{g^tv\mid t\ge0\}$ and $\{g^tv\mid t\le 0\}$, then, Lemma \ref{lem:lift} shows that for all $\epsilon<\epsilon_0$, the following holds:
    $$\begin{cases}W^s_\epsilon\big(g^{\tau^s}(v)\big)\subset W_\epsilon^{ws}\big(g^{\tau^s}(v)\big)=W_\epsilon^{ws}(v)\\ W^u_\epsilon\big(g^{-\tau^u}(v)\big)\subset W_\epsilon^{wu}\big(g^{-\tau^u}(v)\big)=W_\epsilon^{wu}(v)\end{cases}$$
    
    Hence, the sets $W_\epsilon^{ws}(v)$ and $W_\epsilon^{wu}(v)$ intersect $\ONF$ for $\epsilon<\epsilon_0$. According to Equation \ref{eq:inclusion}, this implies that $v$ is weakly regular.

In the case where $M$ is compact, the injectivity radius of $T^1M$ is bounded below globally. Hence, it is enough to assume that for all $\epsilon>0$, there exist two non-negative numbers $\tau^s,\tau^u\ge0$ such that both $W^s_\epsilon\big(g^{\tau^s}(v)\big)$ and $W^u_\epsilon\big(g^{-\tau^u}(v)\big)$ intersect $\ONF$. This proves that the version above of Definition \ref{def:weakly_regular} coincides with the version given in the introduction when $M$ is compact.

\medskip

\begin{propPerso}\label{prop:polish}
Let us assume that $\ONF$ is open and that the curvature of $M$ is bounded below by $-b^2$. Then, the set of weakly regular vectors $\A$ is a $G_\delta$ subset of $T^1M$, hence a Polish space.
\end{propPerso}

\begin{proof}Let $\epsilon>0$ and set $\epsilon'=\frac\epsilon{\sqrt{1+b^2}}$. We will use the notation $W_\epsilon^{wu}(\ONF)$ to denote the union of all the weak $\epsilon$-local unstable sets of the vectors in $\ONF$, as it is done in Equation \ref{eq:intersection}. 

Let $v\in W_{\epsilon'}^{wu}(\ONF)$. Then, there exist $u\in\ONF$, $\sigma\ge0$ and two lifts $\tilde u, \tilde v$ of $u, v$ such that $\tilde u\in W_{\epsilon'}^u(g^{-\sigma} \tilde v)$. Lemma \ref{lem:ELPS} applied to the vector $g^{-\sigma} v$ and $\epsilon>0$ yields the existence of a positive number $\delta'>0$ such that for all $w\in B_{\delta'}(g^{-\sigma} v)$ there exists a lift of $w$ whose $\epsilon$-local unstable set intersects $\mathrm{Proj}^{-1}(\ONF)$, which implies that $w$ belongs to $W_\epsilon^{wu}(\ONF)$.

Moreover, the continuity of the mapping $g^{-\sigma}$ yields the existence of $\delta>0$ such that:
$$\forall v'\in B_\delta(v), \quad g^{-\sigma} v'\in B_{\delta'}(g^{-\sigma} v)$$

Then, for all $v'\in B_\delta(v)$, the vector $g^{-\sigma} v'$ belongs to $W_\epsilon^{wu}(\ONF)$, and since the latter is invariant by the geodesic flow, the ball $B_\delta(v)$ is contained in $W_\epsilon^{wu}(\ONF)$. We have thus proved the following inclusion, where the notation $\mathrm{int}(\cdot)$ denotes the interior of a set: 
$$W_{\epsilon'}^{wu}(\ONF)\subset\mathrm{int}\big(W_\epsilon^{wu}(\ONF)\big)$$

A similar argument would yield the same inclusion for the weak local stable set. Hence, according to Equation \ref{eq:intersection}, one has:
$$\A = \bigcap_{\epsilon>0}\big(\mathrm{int}\big(W_\epsilon^{ws}(\ONF)\big)\cap \mathrm{int}\big(W_\epsilon^{wu}(\ONF)\big)\big)$$

Equation \ref{eq:inclusion} shows that this intersection can be written as a countable intersection of open sets, making $\A$ a $G_\delta$ subset of $T^1M$. The latter being a Polish space, this makes $\A$ a Polish space.
\end{proof}

\medskip

The proof of the following Lemma is based on several dynamical properties of the geodesic flow restricted to $\ONF$ that were proved by Coudène and Schapira \cite{CS14}. It will enable us to prove the $P_{NF}$-closeability of $\A$.

\begin{lem}\label{lem:equiv}Let us assume that the curvature of $M$ is bounded below by $-b^2$, and that $\ONF$ is open and contains at least three distinct periodic orbits. For all $v\in \A$ and $\epsilon>0$, the sets $W_\epsilon^{ws}(v)$ and $W_\epsilon^{wu}(v)$ are dense in $\ONF$.
\end{lem}
\begin{rem}
    One can also show that for all $v\in\A$ and $\epsilon>0$, the set $W_\epsilon^{ws}(v)\cap W_\epsilon^{wu}(v)$ intersects $\ONF$. This result is not needed for our purposes, but can be seen as an alternative definition of $\A$. It can be proved as a corollary of Lemma \ref{lem:equiv} using the local product structure of $\ONF$.
\end{rem}

\begin{proof}
Let $v\in\A$ and $\epsilon>0$. Since $W_{\frac\epsilon2}^{ws}(v)$ intersects the set $\ONF$, which is invariant by the geodesic flow, we can find $\sigma\ge0$, $u\in\ONF$ and two lifts $\tilde u, \tilde v$ of $u, v$ such that $\tilde u\in W^s_{\frac\epsilon2}(g^{\sigma}\tilde v)$. 

Let us pick $w\in\ONF$ and $\delta>0$ and prove that $W_\epsilon^{ws}(v)$ intersects the ball $B_\delta(w)$. Using a similar argument, one could prove that $B_\delta(w)$ intersects also $W_\epsilon^{wu}(v)$.

\medskip

Since $u\in \ONF$ and $\tilde u \in W^u_{\frac{\min(\epsilon, \delta)}{2\sqrt{1+b^2}}}(\tilde u)$, Lemma \ref{lem:ELPS} shows that there exists a positive real number $\kappa>0$ such that for all $\tilde u'\in B_{\kappa}(\tilde u)$, the intersection $W^{s}_{\frac\epsilon2}(\tilde u)\cap W_{\frac\delta2}^u(\tilde u')\cap \mathrm{Proj}^{-1}(\ONF)$ is non-empty. Moreover, the geodesic flow is transitive in restriction to $\ONF$ \cite[Lemma 4.7]{CS14}, thus we can find a vector $u'\in B_\kappa(u)$ and a time $T>0$ such that the vector $g^{-T}u'$ belongs to $B_{\frac\delta2}(w)$. It admits a lift $\tilde u'$ that belongs to $B_\kappa(\tilde u)$. Hence, the definition of $\kappa$ yields the existence of a vector $u''\in\ONF$ that has a lift $\tilde u''$ that belongs to $W^{s}_{\frac\epsilon2}(\tilde u)\cap W_{\frac\delta2}^u(\tilde u')$. 

In particular, $u''$ belongs to $W_{\frac\delta2}^u(u')$, thus:

$$d\big(w, g^{-T}u''\big) \leq d\big(w, g^{-T}u'\big) + d\big(g^{-T}u', g^{-T}u''\big) <\delta$$

Hence, $g^{-T}u''\in B_\delta(w)$. Finally, since $\tilde u''\in W_{\frac\epsilon2}^s(\tilde u)$ and $\tilde u\in W^s_{\frac\epsilon2}(g^{\sigma}\tilde v)$, one has: $\tilde u''\in W_\epsilon^s(g^{\sigma}\tilde v)$. Hence, the vector $g^{-T}u''$ belongs to $W_\epsilon^{ws}(v)$.
\end{proof}

\subsection{Genericity of the set of ergodic measures} 

\begin{defin}Let $\epsilon>0$. A vector $v\in T^1M$ will be said to be $\epsilon$-$P_{NF}$-closeable if for all $T_0>0$ there exist two positive real numbers $T_{per}$ and $T$ and a vector $v_{per}\in P_{NF}$ of period $T_{per}$, satisfying:
$$\begin{cases}\forall t \in [0, T],\quad d(g^{t}v_{per}, g^{t}v)<\epsilon\\
T_0 < T < T_{per} < (1+\epsilon) T
\end{cases}$$
We denote the set of $\epsilon$-$P_{NF}$-closeable vectors by $C_\epsilon$.
\end{defin}

\begin{rem}
    A vector is $P_{NF}$-closeable if and only if it is $\epsilon$-$P_{NF}$-closeable for all $\epsilon>0$.
\end{rem}

\begin{lem}\label{lem:closeability}Let us assume that $\ONF$ is open and contains at least three periodic orbits, and the curvature of $M$ is bounded below by $-b^2$. Let $v\in \A$ be recurrent. Then, for all $\epsilon>0$, there exists $\tau'\ge0$ such that $g^{\tau'}v$ belongs to $C_\epsilon$.
\end{lem}

\begin{proof}
Choose any vector $w\in\ONF $. Using the Closing Lemma \cite[Lemma 4.6]{CS14} near $w$ in restriction to $\ONF$, we can find $\kappa>0$ and $T_{closing}>0$ such that for all $u\in\ONF $ and $t\geq T_{closing}$ with $d(u,w)<\kappa$ and $d(g^{t}u,w)<\kappa$, there exists a periodic vector $v_{per}\in\ONF $ of period $T_{per}\in (t-\epsilon,t+\epsilon)$ such that the orbit of $v_{per}$ is $\frac\epsilon3$-close to the orbit of $u$ in restriction to the interval $[0,T_{per}]$.

\medskip

Let $\delta = \frac1{\sqrt{1+b^2}}\min\left(\frac\kappa2, \frac\epsilon3\right)$.

By Lemma \ref{lem:equiv}, we can find a vector $u\in \ONF\cap B_{\frac\kappa2}(w)$, two non-negative times $\tau,\sigma\ge 0$, and two lifts $\tilde u, \tilde v$ of $u, v$ such that the vector $g^{-\sigma}\tilde u$ belongs to $W^u_{\frac\delta{\sqrt{1+b^2}}}(g^{-\tau}\tilde v)$. Then, according to Lemma \ref{lem:ELPS}, there exists $\delta'>0$ such that:

$$\forall u'\in B_{\delta'}(g^{-\tau} v), \quad W_{\delta}^u( u')\cap W_{\delta}^s(g^{-\sigma} u)\cap\ONF\neq\emptyset$$

\medskip

Let $\delta'' = \min\big(\delta ,\frac{\delta'}2\big)$.

Again, Lemma \ref{lem:equiv} enables us to find a vector $u'\in\ONF\cap B_{\frac\kappa2}(w)$ and two non-negative times $\tau',\sigma'\ge0$ such that the vector $g^{\sigma'} u'$ belongs to $W^s_{\delta''}\big(g^{\tau'}v\big)$.

\medskip

Let us prove that $g^{\tau'}v\in C_\epsilon$.

Since $v$ is recurrent, the vector $g^{-\tau} v$ is also recurrent, thus there exists $T_{rec}\in\R$ such that:
\begin{equation}\label{eq:rec}\begin{cases}
T_{rec}>\max\Big(\frac{\sigma+\sigma'+2\epsilon}\epsilon+\epsilon,~ T_0+\epsilon,~ T_{closing} - \sigma' - \sigma\Big)\\
d\big(g^{T_{rec}+\tau'}v,g^{-\tau}v\big)<\frac{\delta'}2\end{cases}
\end{equation}

One has in particular $T_{rec}\ge0$, which implies that the vector $g^{T_{rec}+\sigma'}u'$ belongs to $W^s_{\delta''}\big(g^{T_{rec}+\tau'}v\big)$, which is contained in the ball of radius $\delta'$ centered on $g^{-\tau}v$. Hence:
$$W_{\delta}^u(g^{T_{rec}+\sigma'}u')\cap W_{\delta}^s(g^{-\sigma}u)\cap\ONF\neq\emptyset$$

Let us pick $\hat u$ in that set. In particular, one has:
$$\begin{cases}
d\big(g^{-T_{rec}-\sigma'}\hat u, u'\big) <\delta \le \frac\kappa2\\
d\big(g^\sigma\hat u, u\big) <\delta \le \frac\kappa2\end{cases}$$

Hence, with the notation $T_{app} = T_{rec}+\sigma'+\sigma$:

$$\begin{cases}T_{app} >T_{closing}\\
d\big(g^{-T_{rec}-\sigma'}\hat u, w\big) ~\le~ d\big(g^{-T_{rec}-\sigma'}\hat u, u'\big) + d(u', w) ~<~ \kappa\\d\big(g^{T_{app} -T_{rec}-\sigma'}\hat u, w\big) ~\le~ d\big(g^{T_{app} -T_{rec}-\sigma'}\hat u, u\big) + d(u, w) ~<~\kappa\end{cases}$$

Hence, the definition of $\delta$ and $T_{closing}$ gives us the existence of a periodic vector $v_{per}\in P_{NF}$ of period $T_{per}\in(T_{app} -\epsilon, T_{app} +\epsilon)$ whose orbit is $\frac\epsilon3$-close to the orbit of $g^{-T_{rec}-\sigma'}\hat u$ in restriction to $[0, T_{per}]$.

\medskip

The inequalities 
$$T_0 + \epsilon < T_{rec} \le T_{app} , \qquad\text{and}\qquad T_{app} +\epsilon < (1+\epsilon)(T_{rec}-\epsilon)$$
stem from Equation \ref{eq:rec}. Let $T=T_{rec}-\epsilon$. Since $T_{per}$ belongs to $(T_{app} -\epsilon, T_{app} +\epsilon)$, one has:
\begin{equation}\label{eq:T_per}T_0 < T < T_{per} < (1+\epsilon) T\end{equation}

The interval $[0, T_{per}]$ contains $[\sigma', T+\sigma']$, hence:
\begin{equation}\label{eq:v_per}
\forall t\in[0, T], \quad d\big(g^{t+\sigma'}v_{per}, g^{t-T_{rec}}\hat u\big)<\frac\epsilon3
\end{equation}

Moreover, the fact that $\hat u \in W_\delta^u(g^{T_{rec}+\sigma'}u')$ and the inequality $T\le T_{rec}$, show that:

\begin{equation}\label{eq:u_hat_bis}\forall t\le T, \quad d\big(g^{t-T_{rec}}\hat u, g^{t+\sigma'}u'\big)<\delta \le \frac\epsilon3
\end{equation}

Similarly, the fact that $g^{\sigma'}u'\in W^s_{\delta''}\big(g^{\tau'}v\big)$ shows that: 

\begin{equation}\label{eq:u_hat}\forall t\ge0, \quad d\big(g^{t+\sigma'}u', g^{t+\tau'}v\big)<\delta'' \le \frac\epsilon3
\end{equation}

Equations \ref{eq:v_per}, \ref{eq:u_hat}, and \ref{eq:u_hat_bis} show that for all $t\in [0, T]$, one has:
$$\begin{aligned}d\big(g^{t+\sigma'}v_{per}, g^{t+\tau'}v\big) ~<~& d\big(g^{t+\sigma'}v_{per}, g^{t-T_{rec}}\hat u\big)+ d\big(g^{t-T_{rec}}\hat u, g^{t+\sigma'}u'\big) \\&+d\big(g^{t+\sigma'}u', g^{t+\tau'}v\big) \\<~& \epsilon\end{aligned}$$

Since $g^{\sigma'}v_{per}$ is $T_{per}$-periodic and Equation \ref{eq:T_per} holds, the proof is complete.
\end{proof}

Let us recall that Gelfert and Kwietniak define the $K$-closeability of a dynamical system $X$ (with $K\subset X$ a set of periodic points), as the fact that all ergodic probability measures on $X$ have a generic point that is $K$-closeable.

\begin{propPerso}\label{prop:closeability}
    The set $\A$ is $P_{NF}$-closeable.
\end{propPerso}

\begin{proof}
Let $\mu\in \mathcal{M}^1_e(\A)$, and let us prove that it has a generic vector that is $P_{NF}$-closeable. Since $\mu$-almost every vector is generic by the Birkhoff Theorem, it is enough to prove that $\mu$-almost every vector is $P_{NF}$-closeable.

By Poincaré recurrence theorem, $\mu$-almost every vector is recurrent. According to Lemma \ref{lem:closeability}, this implies that within the set of vectors $v\in T^1M$ such that both $W_\epsilon^{ws}(v)$ and $W_\epsilon^{wu}(v)$ intersect $\ONF$ for all $\epsilon>0$, $\mu$-almost every vector belongs to $g^{-\tau}(C_\epsilon)$ for some $\tau\ge0$. Since $\mu(\A)=1$, we have proved:
$$\mu\Big(\bigcup_{\tau\ge0} g^{-\tau}(C_\epsilon)\Big)=1$$

For all $\epsilon>0$ and $\tau\ge0$, one has $\mu\big(g^{-\tau}(C_\epsilon)\big) = \mu(C_\epsilon)$. Moreover, one can easily verify that the set $C_\epsilon$ is forward-invariant by the geodesic flow, which implies that for all $\tau_1, \tau_2>0$ such that $\tau_1 < \tau_2$, the inclusion $g^{-\tau_1}(C_\epsilon)\subset g^{-\tau_2}(C_\epsilon)$ holds. Hence:
$$\mu(C_\epsilon)~\le~\mu\Big(\bigcup_{\tau\ge0}g^{-\tau}(C_\epsilon)\Big)~=~1$$

Hence, $\mu(C_\epsilon)=1$, and since the set of all $P_{NF}$-closeable vectors is exactly the intersection of all $C_\epsilon$, for $\epsilon>0$, its measure is also equal to $1$.
\end{proof}

Proposition \ref{gdirac} is a direct corollary of Proposition \ref{prop:polish}, Proposition \ref{prop:closeability}, and the linkability of $P_{NF}$, which was proved in \cite[Section 4.3]{CS14}. The bridge between those properties and the conclusion of Proposition \ref{gdirac} is \cite[Theorem 1.1]{GK16}. Even though \cite[Theorem 1.1]{GK16} is proved in the case of discrete-time dynamical systems, the authors note in \cite[Section 8.2]{GK16} that it can be generalized to flows.

\begin{propB}Let $M$ be a connected, geodesically complete manifold whose curvature is non-positive and bounded below by a negative constant, such that $\ONF$ is open in $T^1M$ and contains at least three periodic orbits. Then $\mathcal{M}^1_e(\A)$ is a generic subset of $\mathcal{M}^1(\A)$ and the following inclusion holds:
$$\mathcal{M}^1(\A) \subset \overline{\mathcal{M}^1_{per}(P_{NF})}$$ 
\end{propB}

\begin{rem}
    With the assumptions of \cite[Theorem 1.1]{GK16}, Gelfert and Kwietniak prove several additional results, involving notably the notion of \textit{measure center} and \textit{strongly mixing measures}. These results hold also for $\A$, as corollaries of \cite[Theorem 1.1]{GK16}, \cite[Section 4.3]{CS14}, and Proposition \ref{prop:closeability}. We do not mention them here because our main focus is on the genericity of ergodic measures.
\end{rem}

\section{Application to notable manifolds}\label{sec:application}

\subsection{A manifold with an embedded flat torus where ergodic measures are dense in the space of invariant probability measures}

In this section, we use Proposition \ref{gdirac} to show that ergodic measures are generic in the space of invariant probability measures on the unit tangent bundle of a non-positively curved manifold that was introduced by Heintze and studied by numerous authors \cite{BBE85, BCFT18}. This manifold is particularly interesting because it contains a totally geodesic embedded flat torus, showing that the existence of flat strips in a manifold of dimension at least $3$ does not prevent the genericity of ergodic measures.

\begin{propC}There exists a compact, non-positively curved manifold $M$ of dimension $3$ that contains a totally geodesic flat torus such that the set of ergodic measures is generic in the space of invariant probability measures.\end{propC}

\begin{proof}Let $M_{hyp}$ be a compact manifold of constant negative curvature with exactly one boundary component that is isometric to a flat torus. Such a manifold can be constructed by truncating a cusp from a negatively curved manifold with one cusp. The metric of $M_{hyp}$ can be modified near its boundary so that the curvature remains negative everywhere except on the boundary, where it is zero, and the second fundamental form of the boundary vanishes, which implies that it is totally geodesic. This exact construction can be found in \cite{BCFT18} and is illustrated in Figure \ref{fig:Heintze}. It is then possible to glue the resulting manifold together with another copy of itself along their boundaries. The resulting manifold will be denoted by $M$. It contains an embedded flat torus $T$, and all sectional curvatures are negative at every point of $M\setminus T$.

\medskip
\begin{figure}[tbp]
    \centering
\begin{tikzpicture}[scale=0.9, thick]

    \draw (-3.26, 1.73) arc (35:325:1.2 and 3);
    \draw[dashed] (-3.26, -1.73) arc (325:395:1.2 and 3);
    
    \draw (-3.6, 2.0) .. controls (-3.0, 1.4) and (-1.5, 0.65) .. (-1, 0.7);
    \draw (-3.6, -2.0) .. controls (-3.0, -1.4) and (-1.5, -0.65) .. (-1, -0.7);
    
    \draw (-1, 0.7) arc (90:450:0.4 and 0.7);
    
    \draw (-1, 0.3) to[bend left=30] (-1, -0.3);
    \draw (-0.95, 0.4) to[bend right=30] (-0.95, -0.4);
    
    \node at (-0.9, 0.9) {$K=0$};
    \node at (-2, 0) {$K<0$};

    \draw[<->, >=stealth] (-0.5, 0) -- (0.5, 0);

    \draw (1, 0.7) arc (90:270:0.4 and 0.7);
    \draw[dashed] (1, -0.7) arc (-90:90:0.4 and 0.7);
    
    \draw[dashed] (1, 0.3) to[bend left=30] (1, -0.3);
    \draw[dashed] (1, 0.4) to[bend right=30] (1, -0.4);
    
    \node at (0.85, 0.9) {$K=0$};
    \node at (2, 0) {$K<0$};

    \draw (1, 0.7) .. controls (1.5, 0.65) and (3.0, 1.4) .. (3.6, 2.0);
    \draw (1, -0.7) .. controls (1.5, -0.65) and (3.0, -1.4) .. (3.6, -2.0);
    

    \draw (4.5, 0) ellipse (1.2 and 3);

\end{tikzpicture}
\caption{Gluing of two symmetric copies of a $3$-dimensional non-positively curved manifold with a boundary component isometric to a flat torus.} 
    \label{fig:Heintze}
\end{figure}

The set $\ONF$ is the set of all vectors that are not tangent to $T$, thus it is open. Moreover, $M_{hyp}$ contains at least one closed geodesic, thus $M$ contains at least two closed geodesics that do not bound a flat strip, which implies that $\ONF$ contains at least four periodic orbits of the geodesic flow.

Let $v\in T^1M$. If its orbit is not contained in $T$, then $v$ belongs to $\ONF$ since the curvature of $M$ is negative outside of $T$. Now, let us assume that $v$ is tangent to $T$, and let $S$ be the totally geodesic surface immersed in $M$ generated by $v$ and the normal vector to $T$ with base-point $\pi(v)$. Then $v$ does not bound a flat strip within $S$. Hence, every weak $\epsilon$-local stable (resp. unstable) set of $v$ contains a vector $w^s$ (resp. $w^u$) such that $\pi(w^s)\in S\setminus T$ (resp. $\pi(w^u)\in S\setminus T$). In particular, the vectors $w^s$ and $w^u$ belong to $\ONF$. Hence, the manifold $M$ being compact, $v$ is weakly regular. We have proved:

$$\A = T^1M$$

According to Proposition \ref{gdirac}, the set of ergodic measures is generic in the space of invariant probability measures on the unit tangent bundle of $M$.\end{proof}

\begin{rem}
    The Heintze manifold, as it is defined in the proof of Proposition \ref{prop:C}, was already known to exhibit ergodic properties similar to those of negatively curved manifolds. Burns, Climenhaga, Fisher and Thompson suggest a construction that allows them to bound the Lyapunov exponent of the orbits tangent to the central torus by controlling how the metric is flattened in a neighborhood of the central flat torus \cite{BCFT18}. Then, assuming a certain lower bound on those Lyapunov exponents, they prove that every scalar multiple of the geometric potential on the Heintze manifold yields a unique equilibrium state, which is fully supported and coincides with the Liouville measure when the scalar multiple is $1$. It is worth noting that the genericity of ergodic measures holds without any lower bound assumption on the Lyapunov exponent of the orbits tangent to the central torus.
\end{rem}

\subsection{The Gromov manifold}
\label{sec:ex}

In this section, the results of Section \ref{sec:cdirac} are applied to a specific example of a non-positively curved manifold. The construction of this manifold is usually attributed to Gromov \cite{G78}, and has been studied by numerous authors \cite{E80, BGS85, K98, GS14}. We begin with a precise description of the Gromov manifold. Then we decompose its unit tangent bundle into invariant sets where the behavior of ergodic measures is easily described.

\medskip

Let $\Sigma_{hyp}$ be a hyperbolic compact surface of genus two. Let $A\subset \Sigma_{hyp}$ be a simple closed geodesic such that $\Sigma_{hyp}\setminus A$ has two connected components isometric to each other and homeomorphic to a torus with one boundary component. We can modify the Riemannian metric of $\Sigma_{hyp}$ so that the curvature vanishes on $A$, and is negative everywhere else: let us denote by $\Sigma$ the compact surface obtained after modifying the metric. 

Let $\Sigma_{\mathrm{left}}$ and $\Sigma_{\mathrm{right}}$ be the two connected components of $\Sigma\setminus A$ and let $A_{\mathrm{left}}$ and $A_{\mathrm{right}}$ be their respective boundaries within $\Sigma$: they are both isometric to $S^1$. Then, $\overline{\Sigma_\mathrm{left}} = \Sigma_\mathrm{left}\cup A_\mathrm{left}$ and $\overline{\Sigma_\mathrm{right}} = \Sigma_\mathrm{right}\cup A_\mathrm{right}$ are compact surfaces with boundaries. Moreover, $M_{\mathrm{left}} = \overline{\Sigma_\mathrm{left}}\times A_{\mathrm{right}}$ and $M_{\mathrm{right}} = A_{\mathrm{left}} \times \overline{\Sigma_\mathrm{right}}$ are compact $3$-manifolds, and their boundary components are exactly the same torus $A_\mathrm{left}\times A_\mathrm{right}$. Hence, we can glue them together along their boundary, resulting in the adjunction space
$M = M_{\mathrm{left}} \cup M_{\mathrm{right}}$. This construction is illustrated in Figure \ref{fig:manifold_gromov}.

\begin{figure}[tbp]
\centering
\includegraphics[angle=0, scale=0.21]{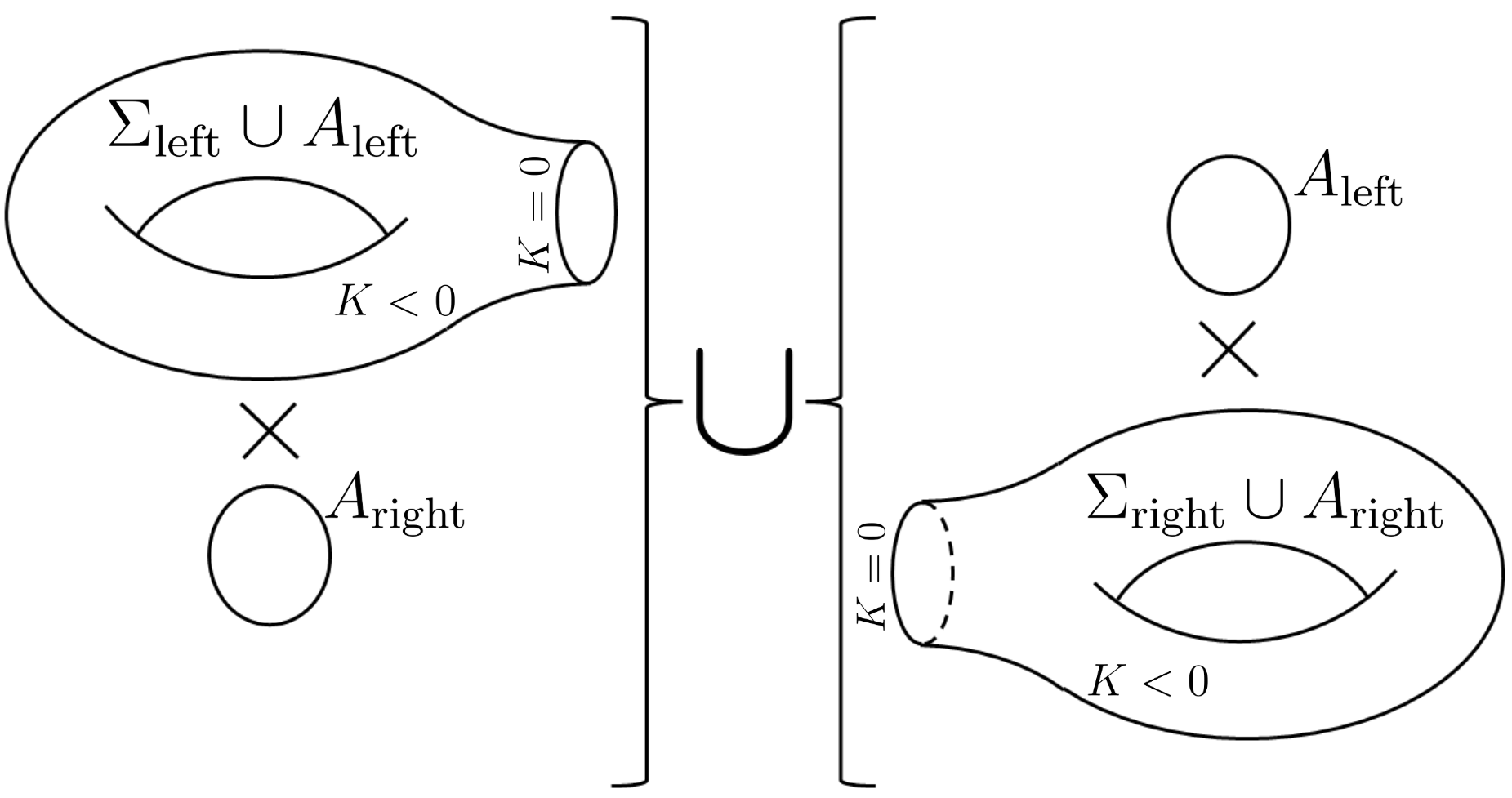}
\caption{Example of a rank-one three-dimensional manifold attributed to Gromov. $M_{\mathrm{left}}$ is represented on the left and $M_{\mathrm{right}}$ is represented on the right. The "$\cup$" symbol represents the gluing along the flat torus $A_{\mathrm{left}}\times A_{\mathrm{right}}$.} 
\label{fig:manifold_gromov}
\end{figure}

The presence of a totally geodesic embedded flat torus implies that $M$ is not negatively curved. On the other hand, Lemma \ref{lem:cross} shows that the set of vectors whose orbits do not bound a flat strip is non-empty and that it is equal to the set of rank-one vectors. Notably, $M$ is thus a rank-one manifold.

We will identify the manifolds $\Sigma_{\mathrm{left}}$, $\Sigma_{\mathrm{right}}$, $A_{\mathrm{left}}$, and $A_{\mathrm{right}}$ with the zero sections of their tangent bundles. In particular, this will enable us to identify $\Sigma_{\mathrm{left}}\times T^1A_{\mathrm{right}}$ and $T^1A_{\mathrm{left}}\times\Sigma_{\mathrm{right}}$ as subsets of $T^1M$.

\begin{lem}\label{lem:cross} Let $\sigma:\mathbb{R}\rightarrow M$ be a unit speed geodesic. Then the following properties are equivalent:
\begin{enumerate}[label=\roman*)]
    \item\label{i} $\sigma$ is a rank-one geodesic
    \item\label{ii} $\sigma$ does not bound a flat strip
    \item\label{iii} $\sigma$ encounters the interiors of both sides of $M$, i.e.,
$$\sigma^{-1}\left(\Sigma_{\mathrm{left}}\times A_\mathrm{right}\right) \neq \emptyset \qquad \text{and} \qquad \sigma^{-1}\left(A_\mathrm{left}\times\Sigma_{\mathrm{right}}\right) \neq \emptyset$$
\end{enumerate}
Moreover, the set of non-weakly regular vectors is exactly:
\begin{equation}\label{eq:Ac}\A^c =  (\Sigma_{\mathrm{left}}\times T^1A_\mathrm{right})\sqcup (T^1A_\mathrm{left}\times \Sigma_{\mathrm{right}})\end{equation}
\end{lem}
\begin{proof} 
\textbf{Step 1: Proof of the equivalence.} Let us assume that \ref{iii} holds. In order to prove \ref{i}, let us assume that there exists a non-zero parallel Jacobi field $J$ along $\sigma$ with $J(t_1)$ orthogonal to $\dot\sigma(t_1)$, and let us find a contradiction.

The set $\sigma^{-1}\left(\Sigma_{\mathrm{left}}\times A_\mathrm{right}\right)$ is an open non-empty subset of $\mathbb{R}$, so it is a disjoint union of open intervals. Let $]t_0, t_1[$ be such an interval, and without loss of generality let us assume that $t_1\ne+\infty$. Then $\sigma(t_1)$ belongs to the central torus $A_{\mathrm{left}}\times A_{\mathrm{right}}$ but $\dot\sigma(t_1)$ is not tangent to it, otherwise the whole geodesic $\sigma$ would be contained in it. Hence, we can find $t_2>t_1$ such that $\sigma(]t_1, t_2[)\subset A_\mathrm{left}\times\Sigma_{\mathrm{right}}$.

Since $J$ is parallel, it is orthogonal to $\dot\sigma$ for all times, and the Jacobi equation can be rewritten as $R(J, \dot\sigma)\dot\sigma=0$. Since $\Sigma_{\mathrm{left}}$ has negative curvature, this implies that the projection of $J$ on $T\Sigma_{\mathrm{left}}\times A_\mathrm{right}$ must vanish in restriction to $]t_0, t_1[$. By continuity of $J$, this implies that $J(t_1)$ is tangent to the foliation induced by $A_{\mathrm{right}}$ on the central torus $A_\mathrm{left}\times A_\mathrm{right}$. Similarly, the projection of $J$ on $A_\mathrm{left}\times T\Sigma_{\mathrm{right}}$ must vanish in restriction to $]t_1, t_2[$, and this implies that $J(t_1)$ is tangent to the foliation induced by $A_{\mathrm{left}}$, which is contradictory, and proves the implication \ref{iii}$\implies$\ref{i}.\\

Now, let us assume that $\sigma$ does not satisfy \ref{iii}. We can distinguish two cases:

\begin{enumerate}[label=\alph*)]
    \item If $\dot\sigma$ is not valued in $(\Sigma_{\mathrm{left}}\times T^1A_\mathrm{right})\cup (T^1A_\mathrm{left}\times \Sigma_{\mathrm{right}})$, then without loss of generality we can decompose $\sigma = (p, \theta)$ with $p:\mathbb{R}\rightarrow \overline{\Sigma_{\mathrm{left}}}$ and $\theta:\mathbb{R}\rightarrow A_{\mathrm{right}}$, such that $\|\dot p(0)\|>0$. Then, the submanifold $p(\mathbb{R})\times A_{\mathrm{right}}$ is a flat strip containing $\sigma$.
    \item If $\dot\sigma$ is valued in $(\Sigma_{\mathrm{left}}\times T^1A_\mathrm{right})\cup (T^1A_\mathrm{left}\times \Sigma_{\mathrm{right}})$, then without loss of generality we can decompose it as $\sigma = (p_0, \theta)$ with $p_0\in \Sigma_\mathrm{left}$ and $\theta : \R\to A_\mathrm{right}$. Since $\Sigma_{\mathrm{left}}$ is a manifold without boundary, there exists $\epsilon>0$ such that the geodesic ball of radius $\epsilon$ centered on $p_0$ is contained in $\Sigma_{\mathrm{left}}$. We will denote this ball by $B_\epsilon(p_0)$. Since $\sigma$ is contained in $B_\epsilon(p_0) \times A_{\mathrm{right}}$, it bounds a flat strip.
\end{enumerate}

We have just proved the implication \ref{ii}$\implies$\ref{iii}. The implication \ref{i}$\implies$\ref{ii} is always true. Indeed, if a geodesic bounds a flat strip, then its rank is at least two.

\medskip

\textbf{Step 2: Description of $\A^c$.} The sets  $\Sigma_{\mathrm{left}}\times T^1A_\mathrm{right}$ and $T^1A_\mathrm{left}\times \Sigma_{\mathrm{right}}$ are disjoint as subsets of the unit tangent bundles of the disjoint submanifolds $\Sigma_{\mathrm{left}}\times A_\mathrm{right}$ and $A_\mathrm{left}\times \Sigma_{\mathrm{right}}$. Let us prove that their union is equal to $\A^c$.

If $\dot\sigma$ is valued in $(\Sigma_{\mathrm{left}}\times T^1A_\mathrm{right})\sqcup (T^1A_\mathrm{left}\times \Sigma_{\mathrm{right}})$, then according to Step 1 (case b), $M$ contains a flat tubular neighborhood of $\sigma$, which must contain the weak $\epsilon$-local stable set of $\dot \sigma(0)$, for $\epsilon$ small enough. Hence, the latter is disjoint from $\ONF$, and we have proved:
$$(\Sigma_{\mathrm{left}}\times T^1A_\mathrm{right})\sqcup (T^1A_\mathrm{left}\times \Sigma_{\mathrm{right}}) \subset \A^c$$

To prove the converse inclusion, let us assume that $\dot\sigma$ is valued in $\A^c$, which according to Step 1 implies that it does not encounter the interior of both sides of $M$. By contradiction, let us assume that $\dot\sigma$ is not valued in $(\Sigma_{\mathrm{left}}\times T^1A_\mathrm{right})\cup (T^1A_\mathrm{left}\times \Sigma_{\mathrm{right}}) \subset \A^c$. As in Step 1, without loss of generality we can decompose $\sigma = (p, \theta)$ with $p:\mathbb{R}\rightarrow \overline{\Sigma_{\mathrm{left}}}$ and $\theta:\mathbb{R}\rightarrow A_{\mathrm{right}}$, such that $\|\dot p(0)\|>0$.

Recall that $\Sigma$ is a surface of genus two and negative curvature everywhere except on $A$. In particular, there are no flat strips in $\Sigma$, thus the geodesic flow on $T^1\Sigma$ is transitive and admits a local-product structure \cite{CS14}. Hence it is possible to find a vector $u_\Sigma\in W^{ss}(\dot p(0))$ with the same norm as $\dot p(0)$ and such that:
$$\forall t> 0, \quad \pi\left(g^tu_\Sigma\right)\in \Sigma_{\mathrm{left}}, \qquad \text{and}\qquad \exists t < 0, \quad \pi\left(g^tu_\Sigma\right)\in \Sigma_{\mathrm{right}}$$

\begin{figure}[tbp]
\centering
\includegraphics[angle=0, scale=0.25]{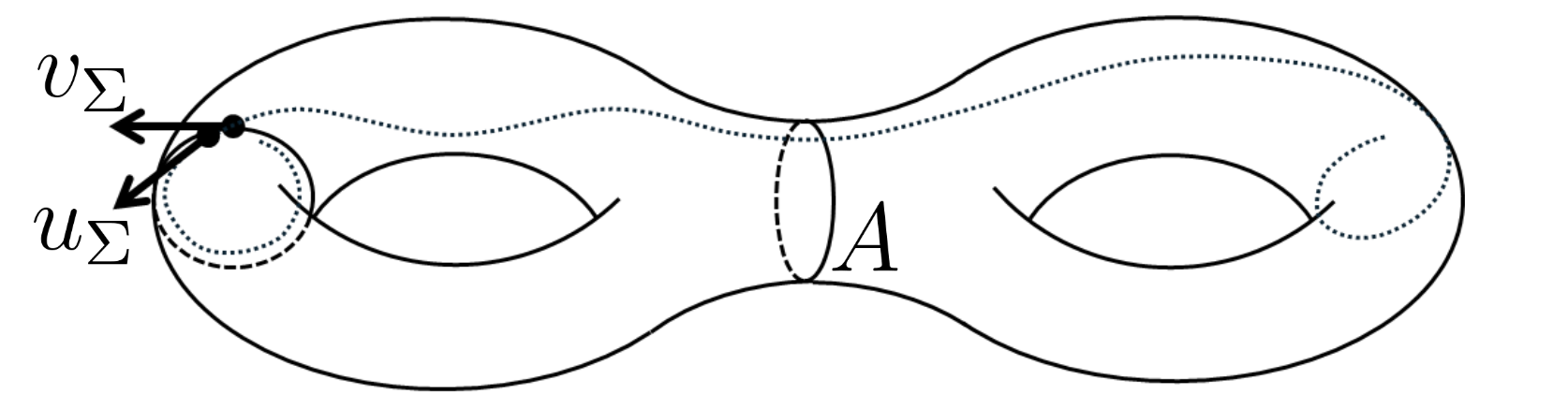}
\caption{Orbit of $u_\Sigma$ and $\dot p$ in $\Sigma$} 
\label{fig:QCS}
\end{figure}

Let $u = u_\Sigma + \dot\theta(0)$. Then $\|u\| = \sqrt{\|\dot p(0)\|^2 + \|\dot\theta(0)\|^2} = 1$ holds and $u$ belongs to the strong stable set of $\dot\sigma(0)$. Moreover, the orbit of $u$ cannot be entirely contained in $\overline{\Sigma_{\mathrm{left}}}\times A_\mathrm{right}$, because the orbit of $u_\Sigma$ intersects $\Sigma_{\mathrm{right}}$, thus, by Step 1 we conclude that $u\in\ONF $. Similarly, one can also exhibit a vector $u'\in \ONF$ in the strong unstable manifold of $\dot\sigma(0)$, proving that $\dot\sigma(0)$ is weakly regular.
\end{proof}

\medskip

The Dirac measure on $\Sigma_{\mathrm{left}}$ (or $\Sigma_{\mathrm{right}}$) supported on a single point $x\in \Sigma_{\mathrm{left}}$ (or $\Sigma_{\mathrm{right}}$) will be denoted by $\delta_x$.
With those notations, the sets $\overline{\text{Conv}\left(\{\delta_x~;~x\in\Sigma_{\mathrm{left}}\}\right)}$ and $\overline{\text{Conv}\left(\{\delta_x~;~x\in\Sigma_{\mathrm{right}}\}\right)}$ are exactly the sets of probability measures on $\Sigma_{\mathrm{left}}$ and $\Sigma_{\mathrm{right}}$ (see the proof of \cite[Proposition 4.4]{P67}). Every ergodic measure on $\A^c$ must have a generic vector by the Birkhoff Ergodic Theorem. Hence, the following description of the set of ergodic probability measures on $\A^c$ stems from Equation \ref{eq:Ac}.

\begin{equation}\label{eq:ergodic_on_Ac_decomposition}
\begin{aligned}\mathcal{M}^1_e(\A^c) = \\\Big(\bigsqcup_{\nu\in \mathcal M^1\big(T^1A_\mathrm{right}\big)}\{\delta_x\otimes \nu\mid x\in\Sigma_{\mathrm{left}}\}\Big)\sqcup\Big(\bigsqcup_{\nu\in \mathcal M^1\big(T^1A_\mathrm{left}\big)}\{\nu\otimes\delta_x\mid x\in\Sigma_{\mathrm{right}}\}\Big)\end{aligned}
\end{equation}

Since $A_{\mathrm{right}}$ and $A_{\mathrm{left}}$ are homeomorphic to $S^1$, their unit tangent bundles contain exactly two connected components, both homeomorphic to $S^1$. Hence, the sets $\Sigma_{\mathrm{left}}\times T^1A_\mathrm{right}$ and $T^1A_\mathrm{left}\times \Sigma_{\mathrm{right}}$ each have two connected components. Since they are open and disjoint, Equation \ref{eq:Ac} shows that $\A^c$ has exactly four connected components.

Moreover, the fact that $A_\mathrm{right}$ and $A_\mathrm{left}$ are homeomorphic to $S^1$ implies that the sets $\mathcal M^1\big(T^1A_\mathrm{right}\big)$ and $\mathcal M^1\big(T^1A_\mathrm{left}\big)$ contain exactly two elements each. Hence, Equation \ref{eq:ergodic_on_Ac_decomposition} shows that $\mathcal{M}^1_e(\A^c)$ has exactly four connected components, each homeomorphic to $\Sigma_{\mathrm{left}}$ (see Figure \ref{fig:top_closure}).

\medskip

Since $A_\mathrm{left}$ is isometric to $A_\mathrm{right}$, every vector in $\A^c$ is periodic with the same period. Hence, every vector in the closure of $\A^c$ is also periodic, and every measure in the closure of $\mathcal M^1_{per}(\A^c)$ is also a Dirac measure, supported on $\overline{\A^c}$:

\begin{equation}\label{eq:closureSc}\overline{\mathcal M^1_{per}(\A^c)}\subset \mathcal M^1_{per}(\overline{\A^c})\end{equation}

The fact that the metric of $M$ splits as a product metric in a tubular neighborhood of every orbit in $\A^c$ shows that the arguments in the proof of \cite[Proposition 1]{M26} can be applied to the Gromov manifold (see also \cite[Corollary 3]{M26}). They show that a Dirac measure supported on $\A^c$ cannot be approximated by Dirac measures supported on $P_{NF}$. Hence:

\begin{equation}\label{eq:M26}\mathcal M^1_{per}(\A^c)\cap \overline{\mathcal{M}_{per}^1(P_{NF})}=\emptyset\end{equation}

Finally, by Proposition \ref{gdirac}, every Dirac measure supported on $\overline{\A^c}\setminus\A^c$ can be approximated by Dirac measures supported on $P_{NF}$. According to Equation \ref{eq:closureSc}, this implies that $\overline{\mathcal M^1_{per}(\A^c)}~\setminus~\mathcal M^1_{per}(\A^c)$ is contained in the closure of the set of Dirac measures on $P_{NF}$. The latter is disjoint from $\mathcal M^1_{per}(\A^c)$ by \ref{eq:M26}, thus the following decomposition holds:

\begin{equation}\label{eq:Sc}\overline{\mathcal M^1_{per}(\A^c)}\subset \mathcal M^1_{per}(\A^c)\sqcup \overline{\mathcal{M}_{per}^1(P_{NF})}\end{equation}

It is now very simple to prove that on the Gromov manifold, Proposition \ref{gdirac} holds as an equivalence. 

\begin{propD} An ergodic probability measure $\mu$ on $T^1 M$ can be approximated by Dirac measures supported on orbits that do not bound a flat strip if and only if $\mu(\A)=1$.\end{propD}
\begin{proof}
According to Proposition \ref{gdirac}, $\mu(\A)=1$ implies that $\mu$ can be approximated by Dirac measures. Conversely, let us assume that $\mu(\A^c)>0$. Then $\mu$ is supported on $\A^c$ because it is ergodic, and it must be a Dirac measure supported on $\A^c$. Finally, Equation \ref{eq:M26} shows that it does not belong to the closure of $\mathcal M^1_{per}(P_{NF})$.\end{proof}

The previous results can be used to decompose the set of ergodic probability measures on the unit tangent bundle of the Gromov manifold, as well as its closure and the space of finite invariant measures, using only Dirac measures.

\begin{cor}\label{cor:erg}
The set of ergodic probability measures on $T^1M$ can be decomposed as: 
$$\mathcal{M}^1_e(T^1M) =\overline{\mathcal{M}^1_{per}(P_{NF})}^{\mathcal{M}^1_e(T^1M)} \sqcup\mathcal{M}^1_{per}(\A^c)$$
where $~^{\overline{~~~}\mathcal{M}^1_e(T^1M)}$ denotes the closure within ${\mathcal{M}^1_e(T^1M)}$. Similarly, the closure of the set of ergodic probability measures on $T^1M$ can be decomposed as:
$$\overline{\mathcal{M}^1_e(T^1M)} = \overline{\mathcal{M}^1_{per}(P_{NF})} \sqcup\mathcal{M}^1_{per}(\A^c)$$
\end{cor}
\begin{proof}
    Since $\A$ and $\A^c$ form an invariant partition of $T^1M$, the sets of ergodic probability measures on $\A$ and $\A^c$ form a partition of the set of ergodic probability measures on $T^1M$. Proposition \ref{gdirac} shows that $\mathcal M_e^1(\A)$ is contained in $\overline{\mathcal M^1_{per}(P_{NF})}$ which is disjoint from $\mathcal M_e^1(\A^c)$ according to Equation \ref{eq:M26}. The decomposition of $\overline{\mathcal M^1_e(T^1M)}$ stems directly from the decomposition of $\mathcal M^1_e(T^1M)$ and Equation \ref{eq:Sc}.
\end{proof}

According to Equation \ref{eq:ergodic_on_Ac_decomposition}, the set of ergodic probability measures on $\A^c$ has four connected components that are homeomorphic to $\Sigma_{\mathrm{left}}$. Moreover, according to Equation \ref{eq:closureSc}, its boundary is exactly $\mathcal{M}^1_{per}\big(\overline{\A^c}\setminus \A^c\big)$, which also has four connected components homeomorphic to $S^1$ and contained in $\overline{\mathcal{M}^1_{per}(P_{NF})}$. As a consequence, $\overline{\mathcal{M}^1_e(T^1M)}$ is homeomorphic to an adjunction space containing $\overline{\mathcal{M}^1_{per}(P_{NF})}$ and four copies of $\Sigma_{\mathrm{left}}$. Notably, the closure of the set of ergodic probability measures is not contractible, illustrating how the behavior of the geodesic flow in non-positive curvature is different from the Anosov setting, where the closure of the set of ergodic measures is equal to the whole space of invariant probability measures, hence convex and contractible.

Notice that $\overline{\mathcal{M}^1_{per}(P_{NF})}$ contains a contractible set, $\mathcal{M}^1(\A)$, but describing exactly the topology of $\overline{\mathcal{M}^1_{per}(P_{NF})}$ remains an open problem.

\begin{figure}[tbp]
\centering
\includegraphics[angle=0, scale=0.22]{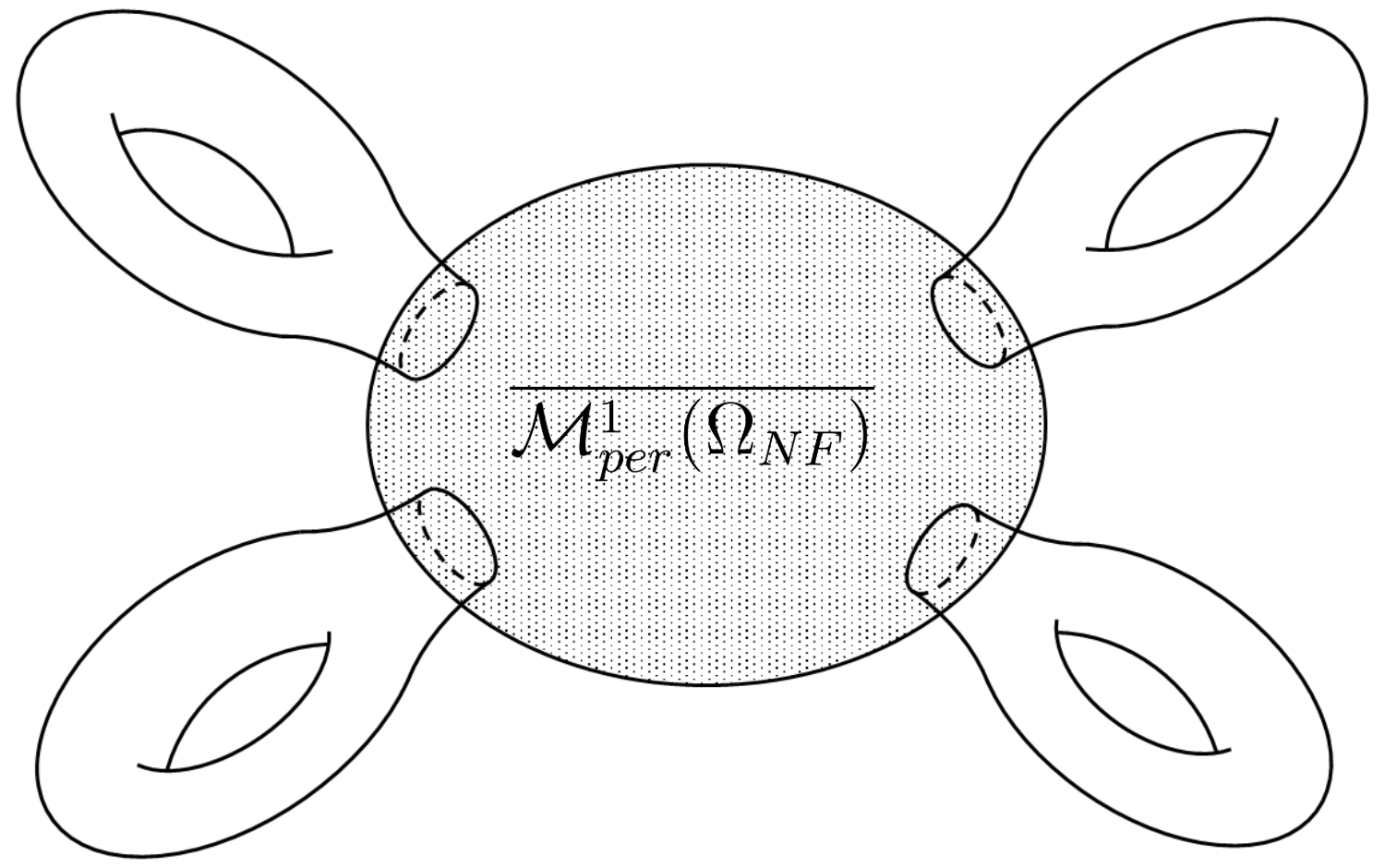}
\caption{Topology of the closure of the set of ergodic measures $\overline{\mathcal{M}^1_e(T^1M)}$.} 
\label{fig:top_closure}
\end{figure}

\begin{propPerso}
Any invariant probability measure $\mu\in\mathcal{M}^1(T^1M)$ can be decomposed as:
$$\mu = \alpha \mu_{NF} + (1-\alpha)\nu$$
with $\mu_{NF}\in\overline{\mathcal{M}^1_{per}(P_{NF})}$, $\nu\in\mathcal{M}^1(\A^c)$ and $\alpha \in[0, 1]$. The measure $\nu$ can be decomposed as a barycenter of ergodic measures on $\A^c$, whose set is described precisely in Equation \ref{eq:ergodic_on_Ac_decomposition}.
\end{propPerso}
\begin{rem}
If we impose $\alpha = \mu(\A)$, this decomposition is unique because $\alpha\mu_{NF}$ is supported on $\A$ and $(1 - \alpha)\nu$ is supported on $\A^c$. However, without this assumption, uniqueness does not hold, since there is no reason for $\alpha\mu_{NF}$ to be supported on $\A$.\end{rem}
\begin{proof}
The space of finite invariant measures on $T^1M$ is the direct sum of the spaces of finite invariant measures on the invariant subspaces $\A$ and $\A^c$, and according to Proposition \ref{gdirac}, up to a scaling factor, every finite invariant measure on $\A$ belongs to the closure of $\mathcal{M}^1_{per}(P_{NF})$.

Moreover, the Ergodic Decomposition Theorem shows that the set of invariant probability measures on $\A^c$ is the closure of the convex hull of $\mathcal{M}^1_e(\A^c)$.
\end{proof}

\vspace{25pt}
\bibliographystyle{amsplain}
\bibliography{sample} 

\end{document}